\documentclass[a4paper,12pt]{amsart}
\usepackage{amsfonts}
\usepackage{enumerate}
\usepackage{amsmath, amsthm, amscd, amssymb}
\usepackage{txfonts}
\usepackage{pdfsync}
\usepackage[title]{appendix}
\usepackage[all]{xy}
\usepackage{latexsym,graphicx}
\usepackage[latin1]{inputenc}
\usepackage{xspace}
\usepackage{array}
\usepackage{newclude}
\usepackage{hyperref}
\usepackage{url}
\usepackage{enumitem}
\hypersetup{
 colorlinks,
 citecolor=black,
 linkcolor=magenta,
 urlcolor=black}
 
\renewcommand*{\HyperDestNameFilter}[1]{\jobname-#1} 
\numberwithin{equation}{section}
\usepackage[centering, marginparwidth=2cm]{geometry}
\usepackage{marginnote}

\addtocontents{toc}{\setcounter{tocdepth}{1}}

\usepackage[nameinlink]{cleveref}

\usepackage{tikz-cd}
\usepackage{leftidx}
\usetikzlibrary{positioning}
\usepackage{dsfont}
\usepackage{tikz}
\usetikzlibrary{matrix,arrows,decorations.pathmorphing,positioning}

\newcommand\blfootnote[1]{%
\begingroup
\renewcommand\thefootnote{}\footnote{#1}%
\addtocounter{footnote}{-1}%
\endgroup
}

\theoremstyle{plain}
\newtheorem{theor}{Theorem}[section]

\newtheorem{conj}[theor]{Conjecture}

\newtheorem{lem}[theor]{Lemma}
\newtheorem{defi}[theor]{Definition}

\newtheorem{prop}[theor]{Proposition}
\newtheorem{cor}[theor]{Corollary}
\theoremstyle{definition}

\newtheorem{rmk}[theor]{Remark}
\newtheorem{question}[theor]{Question}
\theoremstyle{remark}

\newcommand{\RR}{\mathbb{R}}

\newcommand{\NN}{\mathbb{N}}
\newcommand{\ZZ}{\mathbb{Z}}
\newcommand{\QQ}{\mathbb{Q}}

\newcommand{\HH}{{\mathbf H}}
\newcommand{\LL}{{\mathbf L}}

\newcommand{\an}{\textnormal{an}}

\DeclareMathOperator{\PU}{\operatorname{PU}}
\DeclareMathOperator{\SU}{\operatorname{SU}}
\DeclareMathOperator{\U}{\operatorname{U}}
\DeclareMathOperator{\SO}{\operatorname{SO}}
\DeclareMathOperator{\PO}{\operatorname{PO}}

\newcommand{\ad}{\textnormal{ad}}

\newcommand{\MT}{\mathbf{MT}}

\newcommand{\bs}{\backslash}

\newcommand{\Ad}{\operatorname{Ad}}

\DeclareMathOperator{\Hom}{Hom}

\DeclareMathOperator{\Ker}{Ker}

\DeclareMathOperator{\Lie}{Lie}

\DeclareMathOperator{\Gl}{GL}

\newcommand{\Aut}{\operatorname{Aut}}

\newcommand{\codim}{\operatorname{codim}}

\newcommand{\VV}{{\mathbb V}}
\newcommand{\WW}{{\mathbb W}}

\newcommand{\HL}{\textnormal{HL}}
\newcommand{\atyp}{\textnormal{atyp}}

\newcommand{\typ}{\textnormal{typ}}
\newcommand{\pos}{\textnormal{pos}}

\newcommand{\N}{\mathbb{N}}

\newcommand{\Z}{\mathbb{Z}}
\newcommand{\Q}{\mathbb{Q}}

\newcommand{\R}{\mathbb{R}}

\newcommand{\Oo}{\mathcal{O}}

\newcommand{\bS}{\mathbb{S}}

\usepackage[normalem]{ulem}

\newcommand{\C}{\mathbb{C}}

\newcommand{\GL}{\mathbf{GL}}
\newcommand{\SL}{\mathbf{SL}}

\newcommand{\G}{{\mathbf G}}
\newcommand{\bfH}{{\mathbf H}}

\newcommand{\Gam}{\Gamma}
\newcommand{\Lam}{\Lambda}

\usepackage{microtype}
\usepackage{fullpage}

\begin{document}

\title{Rich representations and superrigidity}\date{\today} 
\author{Gregorio Baldi, Nicholas Miller, Matthew Stover, and Emmanuel Ullmo}

\date{\today}

\newcommand{\Addresses}{{
\bigskip
\footnotesize

\textsc{CNRS, IMJ-PRG, Sorbonne Universit\'{e}, 4 place Jussieu, 75005 Paris, France}\par\nopagebreak
\textit{E-mail address}, G.~Baldi: \texttt{baldi@imj-prg.fr}, \texttt{baldi@ihes.fr} \\

\textsc{Department of Mathematics, University of Oklahoma, Norman, OK 73019 }\par\nopagebreak
\textit{E-mail address}, N. Miller: \texttt{nickmbmiller@ou.edu} \\

\textsc{Department of Mathematics, Temple University, Philadelphia, PA 19122}\par\nopagebreak
\textit{E-mail address}, M. Stover: \texttt{mstover@temple.edu} \\

\textsc{I.H.E.S., Universit\'{e} Paris-Saclay, CNRS, Laboratoire Alexandre Grothendieck. 35 Route de Chartres, 91440 Bures-sur-Yvette (France)}\par\nopagebreak
\textit{E-mail address}, E.~Ullmo: \texttt{ullmo@ihes.fr}
}}
\begin{abstract}
We investigate and compare applications of the Zilber--Pink conjecture and dynamical methods to rigidity problems for arithmetic real and complex hyperbolic lattices. Along the way we obtain new general results about reconstructing a variation of Hodge structure from its typical Hodge locus that may be of independent interest. Applications to Siu's immersion problem are also discussed, the most general of which (\Cref{thm:geod2geod}) only requires the hypothesis that infinitely many closed geodesics map to proper totally geodesic subvarieties under the immersion.
\end{abstract}
\maketitle

\tableofcontents
\section{Introduction and Motivation}

A cornerstone in the representation theory of lattices in Lie groups is the Margulis superrigidity theorem \cite[p.\ 2]{margulisbook}. Indeed, the implications of superrigidity are profound, including arithmeticity of irreducible lattices in all noncompact semisimple Lie groups except $\PO(1,n)$ and $\PU(1,n)$, i.e., for all but real and complex hyperbolic lattices. Specifically, Margulis proved superrigidity of irreducible lattices in semisimple Lie groups of real rank at least two. Using the theory of harmonic maps, this was later extended to the rank one groups $\mathrm{Sp}(1,n)$ and $\mathrm{F}_4^{(-20)}$ by Corlette \cite{MR965220} and Gromov--Schoen \cite{MR1215595}. While the harmonic maps method was extended to give a new proof in higher rank \cite{MR1223224}, it is open (and quite relevant in the context of the blend of techniques used in this paper) whether the dynamical methods of Margulis can give an independent proof in the rank one cases.

To be precise, let $\Gam$ be a lattice in a Lie group $G$. A representation $\rho$ from $\Gam$ to a topological group $H$ is said to \emph{extend} if there is a continuous homomorphism $\widetilde{\rho} : G \to H$ so that $\rho$ is the restriction of $\widetilde{\rho}$ to $\Gam$ under its lattice embedding in $G$. Then $\Gam$ is \emph{superrigid} if every unbounded, Zariski dense representation of $\Gam$ to a connected, adjoint simple algebraic group $H$ over a local field $k$ of characteristic zero extends. It is known that there are lattices in $\PO(1,n)$ and $\PU(1,n)$ that are not superrigid; \cite[Sec.\ 5]{zbMATH04086464} for real hyperbolic examples and \cite[Sec.\ 22]{MR586876} or \cite[Sec.\ 1]{KlinglerRigid} for some results in the complex hyperbolic setting. The study of representations of real and complex hyperbolic lattices is very difficult and rich in open questions, and the extent to which superrigidity fails has yet to be made precise. For example, one motivation for this paper is the first part of a question asked by David Fisher \cite[Qtn.\ 3.15]{2020arXiv200302956F}:

\begin{question}[D.\ Fisher]\label{qtn:Fisher}
Let $\Gamma \subset G$ be a lattice where $G= \SO(1,n)$ or $\SU(1,n)$. What conditions on a representation $\rho : \Gamma \to \Gl_m(k)$ imply that $\rho$ extends or almost extends? What conditions on $\Gamma$ imply that $\Gamma$ is arithmetic?
\end{question}

In particular, this paper builds on the techniques from \cite{BFMS, BFMS2, bu} to explore \emph{geometric} conditions on a representation that allow one to use tools from dynamics and/or Hodge theory to prove new rigidity phenomena. The conditions we provide, which address \Cref{qtn:Fisher} through the behavior of totally geodesic submanifolds under maps related to the representation, have also been considered in studying the infinite volume setting; see for example \cite{Hee1, Hee2}. There is some history of success along these lines. For just one family of examples, some classes of representations studied in the literature have \emph{maximal} geometric behavior, where two distinct instances are:
\begin{itemize}

\item Maximal for the \emph{Toledo invariant} (Koziarz--Maubon \cite{MR3612003} and many others)

\item Maximal for their \emph{limit set} (Besson, Courtois and Gallot \cite{MR1738042} and Shalom \cite{zbMATH01545111})

\end{itemize}
Here we study a class of representations generalizing those shown to be superrigid in \cite{BFMS, BFMS2}, namely those that are \emph{geodesically rich}.

\medskip

Before delving into technical definitions and stating our main results, we discuss a primary application of our methods that can be stated much more simply.

Since it arises several times, we briefly recall what it means for $\Gamma \subset G$ to be an arithmetic lattice of the simplest kind. These are well-studied in the literature, so we are terse. In the real hyperbolic case, this means that there is a totally real number field $F$ with integer ring $\mathcal{O}_F$ and a quadratic form $q$ on $F^{n+1}$ so that $\Gam$ is commensurable with $\mathcal{G}(\mathcal{O}_F)$, where $\mathcal{G}$ is the $F$-algebraic group $\PO(q)$. In the complex hyperbolic case, there is a totally imaginary extension $E$ of the totally real field $F$ so that $\Gam$ is commensurable with $\mathcal{G}(\mathcal{O}_F)$, where $\mathcal{G}$ is now $\PU(h)$ for a hermitian form $h$ on $E^{n+1}$ (which is still an $F$-algebraic group). In each case, it is necessarily the case that the form has signature $(1,n)$ at one place of $F$ and is definite at all other places. The associated locally symmetric spaces have an abundance of totally geodesic submanifolds, which are the main object of study in this paper.

\subsection{Siu's immersion problem}

An original motivation for studying the more general rigidity questions considered in this paper is Siu's \emph{immersion problem} \cite[Prob.\ (b), p.\ 182]{siu}. Let $\mathbb{B}^n$ denote the unit ball in $\mathbb{C}^n$ with its Bergman metric. If $\Gam$ is a lattice in the holomorphic isometry group $\PU(1,n)$ of $\mathbb{B}^n$, then $S_\Gam$ will denote the \emph{ball quotient} $\Gam \bs \mathbb{B}^n$.

\begin{question}[Siu, 1985]\label{siuquestion}
Is it true that every holomorphic embedding between compact ball quotients of dimension at least $2$ must have totally geodesic image?
\end{question}

To our knowledge the only previous result is by Cao and Mok \cite{caomok}, which answered \Cref{siuquestion} in the affirmative for $f : \Gam \bs \mathbb{B}^n \to \Lam \bs \mathbb{B}^m$ when $m < 2 n$. See \cite[Thm.\ 1.7]{BFMS2} and the surrounding references for more on Siu's analogous \emph{submersion problem}. The following is our main contribution toward a positive answer to Siu's immersion problem.

\begin{theor}\label{mainthmpre}
Let $f: S_{\Gamma_1} \to S_{\Gamma_2}$ be an immersive holomorphic map between arithmetic ball quotients such that $f(S_{\Gamma_1})$ does not lie in any smaller totally geodesic subvariety of $S_{\Gamma_2}$. Suppose that $S_{\Gam_1}$ contains a proper totally geodesic subvariety of positive dimension. Then the following are equivalent:
\begin{enumerate}
\item The map $f$ is covering.
\item There is a generic sequence of positive-dimensional proper totally geodesic subvarieties $Y_\ell \subset S_{\Gamma_1}$ such that each $f(Y_\ell)$ lies is a proper totally geodesic subvariety $W_\ell \subset S_{\Gamma_2}$.
\end{enumerate}
\end{theor}

The positive-dimensional hypothesis on the subvarieties of $S_{\Gam_1}$ is to rule out the trivial counterexample of points. We offer two proofs of \Cref{mainthmpre}, one using homogeneous dynamics which gives a slightly weaker statement in \Cref{sectionwithratner}, and the second using Hodge theory and functional transcendence in \Cref{sectionimmersion}, as an application of \Cref{mainthmhl} below.

For arithmetic lattices of simplest kind, \Cref{mainthmpre} reduces \Cref{siuquestion} to checking a single dimension.

\begin{cor}
If \Cref{siuquestion} has a positive answer when $S_{\Gam_1}$ is a $2$-dimensional arithmetic ball quotient of simplest type, then it has a positive answer for $S_{\Gam_1}$ an arithmetic ball quotient of simplest type of any dimension $n \ge 2$.
\end{cor}

\begin{rmk}
By \cite{caomok}, the first open case is of a holomorphic embedding $f: S_{\Gamma_1} \to S_{\Gamma_2}$ with $\dim S_{\Gamma_1}=2, \dim S_{\Gamma_2}=4$. Here \Cref{mainthmpre} implies that it suffices to prove that there is an infinite collection of distinct totally geodesic curves $C_i \subset S_{\Gamma_1}$ such that each $f(C_i)$ lies in some strict totally geodesic subvariety $Y_i\subset S_{\Gamma_2}$.
\end{rmk}

There are various similar generalizations that follow from \Cref{richvhs0}. For example, a period map $\psi : S_{\Gamma_1} \to \Gamma \backslash D$ such that the $f^{-1}(Y_i)$ are totally geodesic subvarieties of $S_{\Gamma_1}$ for a generic sequence of Mumford--Tate subdomains $Y_i \subset \Gamma \backslash D$ is necessarily an isomorphism. Here `generic' means that they are not all contained in a finite union of maximal totally geodesic subvarieties of $S$.

\medskip

From a dynamical point of view, the dimension at least two hypothesis in the previous results is natural and critical in that it allows one to access tools from unipotent dynamics. From a Hodge theoretic point of view, one typically works with a complex subvariety. However, there is still something one can say for closed geodesics.

\begin{theor}\label{thm:geod2geod}
Let $S=\Gamma _n \backslash \mathbb{B}^n$ be a ball quotient with an immersive holomorphic map
\begin{displaymath}
f: S \longrightarrow \Gamma _m \backslash \mathbb{B}^m
\end{displaymath}
such that $f(S)$ does not lie in any smaller totally geodesic subvariety of $\Gamma _m \backslash \mathbb{B}^m$. Let $\{C_\ell\}$ be a generic sequence of distinct closed geodesics in $S$ in the sense that they are not all contained in a finite union of maximal totally geodesic subvarieties of $S$. Assume that for each $\ell$ there exits a strict complex totally geodesic subvariety $Y_\ell \subset \Gamma _m \backslash \mathbb{B}^m$ such that $f(C_\ell) \subset Y_\ell$. Then $f$ is a totally geodesic immersion.
\end{theor}

See \Cref{corcor} for the proof. Note that \Cref{thm:geod2geod} applies for nonarithmetic $\Gamma$, since all finite volume ball quotients contain infinitely many free homotopy classes of closed geodesics. That $f$ must take certain closed geodesics into a proper totally geodesic subspace is a very strong hypothesis, but it is not entirely unprecedented. Indeed one can show that some of the maps associated with surjections between complex hyperbolic lattices take a certain (necessarily finite) collection of totally geodesic subspaces of $S_{\Gam_n}$ to totally geodesic subspaces of $S_{\Gam_m}$. We close this section with a final question that would generalize \Cref{thm:geod2geod}.

\begin{question}
Let $\Gamma $ be any lattice in $\PU(1,n)$, $n>1$, and $f: S_\Gamma \to \Gamma_D \backslash D$ be any period map (e.g., $\Gamma_D \backslash D$ is a hermitian symmetric space and $f$ is holomorphic). If $f$ sends a generic sequence of closed geodesics $\{C_\ell\}$ of $S_\Gamma$ to closed geodesics of $\Gamma_D \backslash D$, is then $f(S_\Gamma) $ a totally geodesic subvariety?
\end{question}

\subsection{Rich representations and the main theorem}

Let $\Gamma \subset G$ be a torsion-free arithmetic lattice, where $G$ is $\SU(1,n)$ or $\SO(1,n)$, and let $X$ denote the symmetric space associated with $G$. Then $S_\Gam$ will denote the manifold $\Gam \bs X$, which is a smooth manifold that is also a smooth variety in the complex hyperbolic case\footnote{Since both contexts are relevant to this paper, we use the term submanifold when considering $S_\Gam$ as a differentiable manifold and subvariety when $S_\Gam$ is complex hyperbolic and the submanifold is also a subvariety for the natural algebraic structure.}. We assume that $S_\Gamma$ contains one, and therefore infinitely many, proper totally geodesic submanifolds\footnote{This claim follows easily from the fact that the associated algebraic group commensurates the lattice. In the other direction, if $S_\Gamma$ contains infinitely many maximal totally geodesic submanifolds of dimension at least $2$, then $\Gamma$ is arithmetic; see \cite{BFMS, BFMS2, bu}.} of real dimension at least $2$. 
When this is the case, we say that $S_\Gamma$ (or $\Gam$) is \emph{geodesically rich}. Let $\{S_{\Gamma_i}\}_{i\in \N}$ be a sequence of distinct immersed maximal totally geodesic submanifolds, where maximal means that the only totally geodesic submanifold strictly containing $S_{\Gamma_i}$ is $S_\Gamma$ itself.

The group $\Gamma_i$ is the stabilizer in $\Gam$ of a totally geodesic subspace $Y_i$ of $X$, and the restriction of $\Gam_i$ to $Y_i$ induces a surjective homomorphism from $\Gam_i$ onto the fundamental group of $S_{\Gamma_i}$. The kernel of this homomorphism encodes the action of $\Gam_i$ on the normal bundle to $Y_i$ in $X$. The main purpose of this paper is to introduce the notion of geodesically rich representations, which are morally the ones that ``preserve'' a significant amount of the rich structure of $\Gamma$:

\begin{defi}\label{richdef}
Let $k$ be a local field and $\bfH$ be a semisimple $k$-algebraic group without compact factors. Let $\rho: \Gamma \to \bfH(k)$ be a representation with unbounded, Zariski dense image. We say that $\rho$ is \emph{geodesically rich} (just \emph{rich} when it is clear from context) if the following are satisfied:
\begin{itemize}
\item $\Gamma$ is geodesically rich, and
\item there are infinitely many $i$ so that $\dim \bfH_i < \dim \bfH$, where $\HH_i$ is defined by
\begin{equation}
\mathbf{H}_i\coloneqq\overline{\rho(\Gamma_i)}\subset \bfH
\end{equation}
for each $i$, where closure is with respect to the Zariski topology.
\end{itemize}
\end{defi}

Special consideration will be given in this paper to the case where $S_\Gam$ and the subspaces $S_{\Gam_i}$ are all complex hyperbolic.

\begin{defi}\label{complexrichdef}
Let $\Gamma \subset G=\PU(1,n)$ be a complex hyperbolic lattice and $\rho: \Gamma \to \bfH(k)$ be a rich representation. Then $\rho$ is \emph{complex geodesically rich} if infinitely many of the $S_{\Gamma_i}$ for which $\dim \bfH_i < \dim \bfH$ are complex hyperbolic.
\end{defi}

Broadly, this paper is concerned with identifying situations where the following question has a positive answer.

\begin{question}\label{mainconj}
Do rich representations extend? Equivalently, are there nontrivial rich representations (i.e., other than the case where $\bfH(k) = G$ and $\rho$ is given by conjugation)?
\end{question}

Before stating one of our primary results toward a positive answer, we recall a definition that appears in \cite{zbMATH07623922} that will be used to give a generalization/abstract version of the method of \cite{BFMS, BFMS2}.

\begin{defi}\label{def:comp}
Fix a parabolic subgroup $P\subset G$ with unipotent radical $U$. A pair $(k, \bfH)$ as above is \emph{strongly compatible} with $G$ if for any subgroup $\{1\} \neq \bfH^\prime \subset \bfH$ and every continuous group homomorphism $\tau: P \to (N_\bfH(\bfH^\prime)/\bfH^\prime) (k)$ we have that $U \le \Ker(\tau)$.
\end{defi}

\begin{rmk}\label{remnonarchimedean}
If $k$ is nonarchimedean, then any pair $(k,\bfH)$ is strongly compatible with $G$.
\end{rmk}

Our main theorem establishes \Cref{mainconj} in two cases:

\begin{theor}\label{mainthm12}
Let $\Gam$ be an arithmetic lattice in $G$ which is either $\PO(1,n)$ or $\PU(1,n)$, $\bfH$ be a connected, adjoint algebraic group over a local field $k$ of characteristic zero, and $\rho: \Gamma \to \bfH(k)$ be a rich representation. If either
\begin{enumerate}
\item $\bfH$ is strongly compatible (in the sense of \Cref{def:comp}) and $k$-simple, or
\item $\Gamma$ is complex hyperbolic, $\rho$ is complex geodesically rich, and $\rho$ is cohomologically rigid (i.e., $H^1(\Gamma, \operatorname{Ad}_{\bfH} \circ \rho)=0$),
\end{enumerate}
then $\rho$ extends.
\end{theor}

Notice that the former condition is just on the pair $(k,\bfH)$, whereas the latter depends only on $\rho$. The proof of (1), which follows from the methods developed in \cite{BFMS, BFMS2}, is in \Cref{sectioncomp}. The proof of (2), contained in \Cref{mainthmonvhs}, is related to the study of variations of Hodge structures and in fact it will be related to the so-called \emph{Zilber--Pink conjecture} \cite{2021arXiv210708838B}. Regarding (2), it might be possible to study some nonrigid representations via the so called \emph{factorization theorem}, see for example \cite[Thm.\ 10]{zbMATH00066548} and various more recent generalizations. The fact that $(k, \bfH)$ is always compatible when $k$ is nonarchimedean has the following fairly direct corollary.

\begin{cor}\label{cor:Integral}
Suppose that $\Gam \subset \PU(1,n)$ is an arithmetic lattice and that $\rho : \Gam \to \bfH(k)$ is a rich representation to a $k$-simple algebraic group $\bfH$ defined over a number field $k$. Then $\rho$ is integral.
\end{cor}

\Cref{cor:Integral} should be viewed as a complement to the fundamental recent work of Esnault and Groechenig \cite{MR3874695} that proves integrality of certain cohomologically rigid local systems. Note that \Cref{cor:Integral} does not assume cohomological rigidity, only that the representation is defined over a number field. Moreover, note that there are homomorphisms from complex hyperbolic lattices onto surface groups \cite{zbMATH01960047}, hence not all representations of complex hyperbolic lattices are definable over a number field. As noticed in \cite{BFMS2}, the version of \Cref{cor:Integral} that implicitly follows from the results there makes the main theorem of that paper independent of \cite{MR3874695}.

\begin{rmk}\label{rem:NoGeod}
In contrast with \Cref{thm:geod2geod}, we now give a simple cautionary example regarding generalizing richness to closed geodesics. Let $G$ be $\PO(1,n)$ for $n \ge 3$ or $\PU(1,n)$ for $n \ge 2$, and suppose that $\Gam_1, \Gam_2$ are arithmetic lattices in $G$ of simplest type for which there is a surjection $\rho : \Gam_1 \to \Gam_2$ with infinite kernel; such representations are known to exist in both settings. Choose a maximal subgroup $\Delta \subset \Gam_2$ associated with a proper totally geodesic submanifold of $S_{\Gam_2}$. Then $\rho^{-1}(\Delta)$ contains the kernel of $\rho$, and therefore is Zariski dense in $G$. Consequently, $\rho^{-1}(\Delta)$ contains a profinitely open collection of regular semisimple elements \cite{MR3738090}, hence they preserve no totally geodesic subspace of $X$ except their axis of translation, and thus the associated closed geodesics on $S_\Gam$ are maximal. In the complex hyperbolic setting the associated closed geodesics are \emph{generic} in the sense that their union is Zariski dense in $S_\Gam$, hence this provides a fairly rich collection of $1$-dimensional maximal totally geodesic subspaces.
\end{rmk}

Representations as in \Cref{rem:NoGeod} are unbounded and Zariski dense with connected, adjoint target, but they cannot extend since they have infinite kernel. They have a well-distributed sequence of maximal closed geodesics for which the associated elements in $\rho(\Gam_1)$ all preserve a totally geodesic subspace of the target, hence the representation could be considered rich in analogy with the higher-dimensional setting. While one might reasonably interpret this as saying the $1$-dimensional case is hopeless, \Cref{thm:geod2geod} indicates that stronger geometric hypotheses can still lead to a rigidity theorem.

\medskip

We now comment on how richness may differentiate between real and complex hyperbolic lattices. For example, we do not know the answer to the following: 

\begin{question}\label{qtn:LinearRich}
Is every faithful linear representation of an arithmetic lattice in $\PU(1,n)$ of simplest kind rich?
\end{question}

Bending deformations of lattices in $\PO(1,n)$ \cite[Sec.\ 5]{zbMATH04086464} easily imply that \Cref{qtn:LinearRich} is false for real hyperbolic lattices of simplest kind. However, Goldman--Millson rigidity \cite{MR884798} implies that there are no analogous deformations in the complex hyperbolic setting. It seems that examples giving a negative answer to \Cref{qtn:LinearRich} would come from representations of a new kind. On the other hand, a positive answer would (for example) prove superrigidity of unbounded, Zariski dense, and faithful representations to connected, adjoint targets.

Since \Cref{mainthm12} proves that all rich representations to connected, adjoint $k$-simple groups over nonarchimedean fields of characteristic zero are bounded, all other results in this paper consider the cases where $k$ is $\R$ or $\C$. For the remainder of this introduction we consider rich representations of complex hyperbolic lattices with coefficients in $\R$.

\subsection{Holomorphic and VHS results}
 
To enter in the world of variations of Hodge structures, we focus on \emph{complex hyperbolic lattices} $\Gamma \subset G=\PU(1,n)$ with $n>1$. See \Cref{hodgeprel} for the various notions from Hodge theory used in this subsection. In this setting we can restate \Cref{complexrichdef} in a more geometric way. Suppose for example that $H$ is the automorphism group of a Hermitian symmetric domain $X_H$, $\rho: \Gamma \to H$ takes values in a lattice $\Gamma_H \subset H$, and is induced by taking the induced map on fundamental groups for a holomorphic map
\begin{equation}\label{eq02}
f: S_\Gamma \longrightarrow \Gamma_H\backslash X_H.
\end{equation}
Then the complex richness of $\rho$ translates into the following: $f(S_{\Gamma_i})$ lies in a strict complex totally geodesic submanifold $Y_i \subset \Gamma_H\backslash X_H$ for infinitely many $i$.

We remark here that the setting described in \eqref{eq02} is a special case of the variation of Hodge structure (VHS, from now on) formalism. The next theorem shows that rich representations underlying an integral pure polarized VHS extend.

\begin{theor}\label{richvhs0}
If $\rho$ is a complex rich representation underlying a direct $\C$-factor of an irreducible $\Q$VHS, then $\rho$ extends.
\end{theor}

The proof of \Cref{richvhs0} builds on the main results of \cite{2021arXiv210708838B}, namely the \emph{geometric part of the Zilber--Pink conjecture}. In particular, the proof crucially uses the following, result proved in \Cref{sectionhodge}, which may be of independent interest.

\begin{theor}\label{mainthmhl}
Let $S$ be a smooth quasi-projective complex variety supporting $ \Z$VHSs $\mathbb{V}_1, \mathbb{V}_2$ whose algebraic monodromy groups are $\QQ$-simple and nontrivial. If
\[
\HL(S,\mathbb{V}_1^\otimes)_{\pos, \typ}=\HL(S,\mathbb{V}_2^\otimes)_{\pos, \typ}\neq \emptyset,
\]
then $\mathbb{V}_1, \mathbb{V}_2$ are isogenous.
\end{theor}

Given a smooth quasi-projective complex variety $S$ and $\VV\to S$ a variation of Hodge structures, we write $\HL(S,\mathbb{V}^\otimes)_{\pos, \typ}$ for the so called \emph{typical Hodge locus of positive period dimension}, as introduced in \cite{2021arXiv210708838B}. See \Cref{def:isog} for the precise definition of an isogeny in this context.

\subsection{Outline of the paper}
In \Cref{sectioncomp} we prove the statement about compatibility. The rest of the paper is devoted to the complex hyperbolic case. In \Cref{sectionwithratner} we give the argument based on various forms of Ratner's equidistribution theorem. In \Cref{sectionimmersion} we give a simple self-contained proof of the results about Siu's immersion problem. The rest of the paper uses more ideas and techniques from Hodge theory. In particular \Cref{hodgeprel} is devoted to all the required preliminaries from Hodge theory and the Zilber--Pink paradigm. In \Cref{sectionhodge} we prove the most general statement on isogenies of VHSs, namely \Cref{mainthmhl}. Finally we prove the main results about the extension of rich representations underlying a VHS in \Cref{mainthmonvhs}.
\subsection*{Acknowledgements}
Baldi and Ullmo were partially supported by the NSF grant DMS-1928930, while they were in residence at the MSRI in Berkeley, during the Spring 2023 semester. Miller was partially supported by Grant Number DMS-2005438/2300370 and Stover was partially supported by Grant Number DMS-2203555 from the National Science Foundation. G.B. also thanks the IHES for the excellent working conditions.

\section{Dynamics background and compatibility}\label{sectioncomp}

This section contains some background on the dynamical side, along with the proof of the first part of \Cref{mainthm12}. Let $G$ be either $\PO(1,n)$ or $\PU(1,m)$. We first state some basic facts and establish some notation that will be used throughout this paper. What is described immediately below recaps facts that are given a detailed treatment in \cite{BFMS, BFMS2}.

Fixing a maximal compact subgroup $K$ in $G$, the associated locally symmetric space is $X = G/K$. Fix a representative $Y \subset X$ for each type (i.e., $G$-orbit) of totally geodesic subspace of $X$ with dimension at least two and call such a representative a \emph{standard} subspace. If $W$ is the subgroup of the stabilizer of $Y$ in $G$ generated by unipotent elements, then the full stabilizer of $Y$ is the normalizer $N_G(W)$, which is a compact extension of $W$.

The Haar measure on $G$ will be denoted $\mu_G$, and (after appropriate scaling) given a lattice $\Gam$ in $G$ this induces a probability measure $\mu_{\Gam \bs G}$ on the quotient $\Gam \bs G$. Suppose that the translate $g Y$ of the standard subspace $Y$ has the property that $g L g^{-1} \cap \Gam$ defines a lattice $\Delta$ in $T=gLg^{-1}$ for some $W \subseteq L \subseteq N_G(W)$. One obtains an immersion $S_\Delta \looparrowright S_\Gam$ induced by the embedding
\[
\Gam \bs \Gam g L \hookrightarrow \Gam \bs G
\]
and an associated push-forward of (the $g$-translate of) $\mu_{\Delta \bs T}$ to a probability measure on $\Gam \bs G$. Thus one obtains a homogeneous $W$-invariant measure on $\Gam \bs G$ with proper support. Upon taking an ergodic component of this measure, one can moreover choose $L$ satisfying the above properties for which the corresponding measure is $W$-ergodic. If $\{S_{\Gam_i}\}$ is an infinite sequence of distinct maximal totally geodesic subspaces of $S_\Gam$, then one can extract a subsequence so that they are all associated with $G$-translates of the same $Y$, so $\Gam_i$ is a lattice in $T_i=g_iL_ig_i^{-1}$ with $W\subseteq L_i \subseteq N_G(W)$ for all $i$. The maximality assumption ensures that $\mu_{\Gam \bs G}$ is a weak-$\ast$ limit\footnote{All limits of measures in this paper will be in the weak-$\ast$ sense.} of (the $g_i$ translates of) the measures $\mu_{\Gam_i \bs T_i}$.

\begin{theor}\label{thmrichsecond}
Suppose that $\Gam\subset G$ is a geodesically rich lattice. Let $k$ be a local field of characteristic zero and $\bfH$ be a connected, adjoint $k$-algebraic group so that the pair $(k, \bfH)$ is compatible with $G$. If $\rho: \Gamma \to \bfH(k)$ is a rich representation, then $\rho$ extends.
\end{theor}

\begin{proof}
By \Cref{richdef}, $\rho$ is a homomorphism with unbounded, Zariski dense image. Following Theorem 1.6 and Remark 1.7 of \cite{BFMS}, to prove that there exists a continuous extension ${\widetilde{\rho} : G\to \bfH(k)}$ of $\rho$, it suffices to find
\begin{enumerate}
\item a noncompact almost-simple, proper subgroup $W \subset G$,
\item a $k$-rational faithful, irreducible representation $\bfH \to \SL (V)$ on a finite dimensional $k$-vector space $V$, and
\item a $W$-invariant measure $\nu$ on $\Gam \bs (G \times \mathbb{P}(V))$ that projects to the Haar measure $\mu_{\Gamma \bs G}$ on $\Gamma \bs G$.
\end{enumerate}
We now explain how richness implies the existence of (1)-(3), which follows the argument given in \cite[\S 3]{BFMS} in a similar context.

Recall that by hypothesis, $S_{\Gamma}$ contains infinitely many distinct maximal totally geodesic subspaces $S_{\Gam_i}$. Passing to a subsequence, we assume they are of the same type (in the sense described in this section). This defines a sequence of homogeneous, $W$-ergodic measures $\{\mu_i\}$ with $\mathrm{supp}(\mu_i)\neq \Gam \bs G$ for which $\mu_i \to \mu_{\Gam \bs G}$, where $W \subset G$ is a fixed proper subgroup generated by unipotent elements. This provides the $W$ in item (1) above.

The condition that $\mu_i$ is $W$-ergodic and homogeneous is precisely the condition that there exists $L_i \subset G$ with $W\le L_i\le N_G(W)$ and $g_i\in G$ for which $\Gam \bs \Gam g_i L_i \subset \Gam \bs G$ is closed, and the measure obtained by push-forward is ergodic for the $W$-action. If $T_i=g_i L_i g_i^{-1}$, then $\Gam \bs \Gam g_i L_i$ being closed in $\Gam \bs G$ is equivalent to $\Gamma_i=T_i\cap\Gamma$ being a lattice in $T_i$. The condition that $\mu_i$ is $W$-ergodic is equivalent to the condition that $T_i$ is the identity component of the Zariski closure of $g_i N_G(W) g_i^{-1} \cap\Gamma$.

Returning to the proof at hand, recall from \Cref{richdef} that $\bfH_i=\overline{\rho(\Gamma_i)}$ is a $k$-defined noncompact proper subgroup of $\bfH$ with $\dim(\bfH_i) < \dim(\bfH)$ for infinitely many $i$. Passing to a subsequence, we can assume that $\dim(\bfH_i)$ is constant, and we denote this constant by $d$. Consider the $d^{th}$ exterior power of the adjoint representation
\[
\wedge^d(\Ad):\bfH(k)\longrightarrow\mathbf{GL}\left(\bigwedge\nolimits^d\mathfrak{h}\right),
\]
where $\mathfrak{h}$ is the Lie algebra of $\bfH(k)$. The Lie algebra $\frak{h}_i=\Lie(\bfH_i(k))$ determines a line $\ell_i \subset \bigwedge\nolimits^d\frak{h}$. The representation $\wedge^d(\Ad)$ need not be irreducible, however there exists a nontrivial summand $\theta:\bfH \to \mathbf{GL}(V)$, where $V$ is a $k$-defined vector space, for which infinitely many $\ell_i$ project nontrivially to $V$.
This representation satisfies item (2) above.

Finally for item (3), we again pass to a subsequence to assume that all $\ell_i$ project nontrivially to $V$ and therefore $\theta(\rho(\Gamma_i))$ is contained in the stabilizer of $\ell_i$ for infinitely many $i$. Let $x_i$ be the point corresponding to the image of $\ell_i$ in the projectivization $\mathbb{P}(V)$ and we denote the composition of $\theta\circ\rho$ with the projectivization by $\overline{\theta}$.

The \emph{suspension space}, $\Gam \bs (G\times\mathbb{P}(V))$ is formed via the identification
\[
(g,[v])\sim (\gamma g, \overline{\theta}(\gamma)[v]),
\]
which is naturally a fiber bundle over $\Gam \bs G$. Since the right $G$-action on the first factor of $G\times\mathbb{P}(V)$ commutes with this $\Gamma$-action, it descends to a right $G$-action on the suspension space. Recalling the notation from the beginning of the proof, the fact that $\Gam \bs \Gam T_i$ is closed in $\Gam \bs G$ is equivalent to the natural map $\Gamma_i \bs T_i\to \Gam \bs \Gam T_i$ being a homeomorphism. As $\overline{\theta}(\Gamma_i)$ fixes $\{x_i\}$, it is straightforward that $\Gam_i \bs T_i$ is isomorphic to $\Gam_i \bs (T_i\times \{x_i\})$ under the same action as above. Recalling that $g_i L_i=T_i g_i$, commutativity with the right $G$-action yields the following commutative diagram:
\[
\xymatrix{\Gam_i \bs \left(L_i\times\{x_i\}\right)g_i\ar[r]^-{\cong}\ar[d]^-{\cong}&\Gam_i \bs \left(T_i\times\{x_i\}\right)\ar@{^{(}->}[r]\ar[d]^-{\cong}&\Gam \bs \left(G\times\mathbb{P}(V)\right)\ar[d]\\
\Gam \bs \Gam g_i L_i= \Gam \bs \Gam T_i g_i \ar[r]^-{\cong}&\Gam \bs \Gam T_i\cong \Gam_i \bs T_i\ar@{^{(}->}[r]&\Gam \bs G
}
\]
If $\sigma_i$ denotes the induced map $\sigma_i:\Gam \bs \Gam g_i L_i\hookrightarrow \Gam \bs (G\times\mathbb{P}(V))$, then $\nu_i=(\sigma_i)_*\mu_i$ provides a $W$-invariant, ergodic measure on $\Gam \bs (G\times\mathbb{P}(V))$. Again passing to a subsequence, let $\nu_\infty$ denote any weak-$\ast$ limit of the $\nu_i$. Then any ergodic component of $\nu_\infty$ projects to $\mu_{\Gam \bs G}$ and hence satisfies (3) since projection and weak-$\ast$ limit commute. This completes the proof.
\end{proof}

\section{Dynamical proof for a variant of \Cref{mainthmpre}}\label{sectionwithratner}

The purpose of this section is to give a dynamical proof of the following variant of \Cref{mainthmpre}. We continue with the notation established in \Cref{sectioncomp}.

\begin{theor}
Let $S_n =\Gamma _n \backslash \mathbb{B}^n$ be a smooth complex ball quotient that is arithmetic of simplest kind, and suppose that there is immersive holomorphic map
\begin{displaymath}
f: S_n \longrightarrow S_m \coloneqq \Gamma _m \backslash \mathbb{B}^m
\end{displaymath}
to another arithmetic ball quotient of simplest kind. Suppose that $Y_\ell$ is a sequence of distinct irreducible totally geodesic divisors on $S_m$ for which the top dimensional irreducible components of $f^{-1}(Y_\ell)$ are totally geodesic subvarieties of $S_n$. Then $f$ is a totally geodesic immersion.
\end{theor}

\begin{proof}
Throughout this proof, let $G=\SU(1,n)$.
Let $\omega$ be the K\"{a}hler form on $S_m$ descending from the $G$-invariant form on $\mathbb{B}^m$ and $\alpha$ be the analogous form on $S_n$. We will prove that $f^* \omega = \lambda \alpha$ for some constant $\lambda$, hence the induced metric on $S_n$ is a multiple of the locally symmetric metric. This implies that $f$ must be a totally geodesic immersion and thus, under the standard normalizations, $f^* \omega = \alpha$.

A simple transversality argument implies that $f^{-1}(Y_\ell)$ has top-dimensional irreducible components with codimension one for all but finitely many $\ell$. Let $\beta$ be a positive $(n-1,n-1)$ form on $S_m$ with compact support. The $Y_\ell$ equidistribute in $S_m$, so by applying \cite[Cor. 1.5]{km}\footnote{Also see the related paper of M\"oller--Toledo's \cite{zbMATH06442355} and the more general work of Tayou--Tholozan \cite{zbMATH07643474}.} to both $\beta$ and $\omega^{n-1}$ we deduce that
\begin{equation}\label{eq4.1}
\frac{\int_{f(S_n)\cap Y_\ell} \beta}{\int_{f(S_n)\cap Y_\ell}\omega^{n-1}} \longrightarrow \frac{\int_{f(S_n)}\beta \wedge \omega}{\int_{f(S_n)}\omega^n}
\end{equation}
as $\ell$ goes to infinity. By definition of the pullback, the left hand side of Equation \eqref{eq4.1} is also equal to 
\begin{displaymath}
\frac{\int_{f^{-1}(Y_\ell)} f^*\beta}{\int_{f^{-1}(Y_\ell)}f^*\omega^{n-1}}=\frac{\int_{f^{-1}(Y_\ell)} \alpha^{n-1}}{\int_{f^{-1}(Y_\ell)}f^*\omega^{n-1}} \frac{\int_{f^{-1}(Y_\ell)}f^*\beta}{\int_{f^{-1}(Y_\ell)}\alpha^{n-1}}.
\end{displaymath}
We want to compute the limits of $\frac{\int_{f^{-1}(Y_\ell)}f^*\beta}{\int_{f^{-1}(Y_\ell)}\alpha^{n-1}}$ and $\frac{\int_{f^{-1}(Y_\ell)}f^*\omega^{n-1}}{\int_{f^{-1}(Y_\ell)}\alpha^{n-1}}$.

Since there are possibly several irreducible components, we write
\begin{displaymath}
f^{-1}(Y_\ell)=Z_{\ell,1}\cup .... \cup Z_{\ell,n_\ell}.
\end{displaymath}
Now we use Ratner's theorem in $S_n$, using the fact that $f^{-1}(Y_\ell)$ consists of totally geodesic divisors. Until this point we only used that we had some typical intersections, now we redo the same computations using what we know about $S_n$ and $f^{-1}(Y_\ell)$.

\begin{lem}\label{lem:equidistribute} Let $\{\mu_{f^{-1}(Y_\ell)}\}$ be the sequence
\begin{displaymath}
\mu_{f^{-1}(Y_\ell)}\coloneqq \frac{1}{n_\ell} \sum_{i =1}^{n_\ell} \mu_{Z_{\ell,i}}
\end{displaymath}
of probability measures on $\Gam_n\bs G$. Then the $\mu_{f^{-1}(Y_\ell)}$ weak-$\ast$ converge to $\mu_{S_n}$.
\end{lem}

\begin{proof}
We need to recall the description of the totally geodesic divisors on $S_n=\Gamma_n\backslash \mathbb{B}^n$. We refer to \cite[Lem. 8.2]{BFMS2} for a more detailed description and for a proof of the following facts. We have an isomorphism $\mathbb{B}^n= G/K$, where $K=\mathrm{S}(\U(1)\times \U(n))$ is a maximal compact subgroup of $G$, and a fixed embedding $W\coloneqq\SU(1,n-1)\subset G$. Let $\pi: \Gamma_n\backslash G\longrightarrow S_n$ be the natural projection associated with the right quotient by $K$. Then any totally geodesic divisor $Z$ on $S_n$ is of the form
\[
Z= \pi( \Gamma_n\backslash \Gamma_n g W)
\]
for some $g\in \SU(1,n)$ such that $gWg^{-1}\cap \Gamma_n$ is a lattice in $gWg^{-1}$ (see \Cref{sectioncomp} and \Cref{rem:TGsubgroup} at the end of the proof).

Using this description we can write $Z_{\ell,i}= \pi (\Gamma_n\backslash \Gamma_n g_{\ell,i} W)$ for some $g_{\ell, i}\in \SU(1,n)$, where $1\le i\le n_\ell$ for each $\ell$. Let $\mu_{\ell,i}$ be the $W$-invariant probability measure on 
$\Gamma_n\backslash G$ with support $\Gamma_n\backslash \Gamma_n g_{\ell,i} W$ and define
\[
\mu_\ell\coloneqq \frac{1}{n_\ell}\sum_{i=1}^{n_\ell} \mu_{\ell,i}.
\]
Note that $\pi_*\mu_\ell=\mu_{f^{-1}(Y_\ell)}$ by \cite[Lem. 8.2(1)]{BFMS}. As each $\mu_{\ell,i}$ is $W$-invariant so is $\mu_\ell$ and consequently any weak-$\ast$ limit $\mu$ of a subsequence of the $\mu_\ell$ is a $W$-invariant measure. Moreover, the similar argument to \cite[Lem. 8.3]{BFMS} using quantitative non-divergence shows that there is no escape of mass and thus $\mu$ is also a probability measure.

Let $\nu$ be an ergodic component of $\mu$. By Ratner's theorem \cite{RatnerMeasure}, $\nu$ is a homogeneous, $W$-invariant, $W$-ergodic measure and therefore is either a constant multiple of the Haar measure on $\Gam_n\bs G$ or a constant multiple of the Haar measure on a proper closed $W$-orbit $V\subsetneq \Gam_n\bs G$. However, note that for any such closed $W$-orbit, $\mu_{\ell,i}(V)=0$ for sufficiently large $\ell$ and all $1\le i\le n_\ell$ and hence the latter cannot occur. Therefore only the former possibility occurs and hence $\nu=\mu_{\Gam_n\bs G}$. Since pushforward and weak-$\ast$ limits commute, we conclude that $\mu_{f^{-1}(Y_\ell)}\to\mu_{S_n}$ as required.
\end{proof}

As a consequence, again by upgrading Ratner to a statement about forms as in \cite[Cor. 1.5]{km}, but this time in $S_n$ rather than in $\mathbb{B}^m$, we get
\begin{displaymath}
\frac{\int_{f^{-1}(Y_\ell)}f^*\beta}{\int_{f^{-1}(Y_\ell)}\alpha^{n-1}} \longrightarrow \frac{\int_{S_n}f^*\beta \wedge \alpha}{\int_{S_n} \alpha^n}.
\end{displaymath}
When we apply this result to the case $\beta=\omega^{n-1}$, we find that
\begin{displaymath}
\frac{\int_{f^{-1}(Y_\ell)}f^*\omega^{n-1}}{\int_{f^{-1}(Y_\ell)}\alpha^{n-1}} \longrightarrow \frac{\int_{S_n}f^*\omega^{n-1} \wedge \alpha}{\int_{S_n} \alpha^n}.
\end{displaymath}
Putting everything together, we obtain
\begin{displaymath}
\frac{\int_{S_n} f^* \beta \wedge \alpha}{\int_{S_n} f^*\omega^{n-1} \wedge \alpha} = \frac{\int_{S_n} f^* \beta \wedge f^*\omega}{\int_{S_n} f^* \omega^n}.
\end{displaymath}
Since the above is true for all positive $(n-1,n-1)$ forms $\beta$ with compact support on $S_m$ and $f^*: \Omega^p(S_m)\to \Omega^p(S_n)$ is locally surjective because $f$ is an immersion, we deduce that $f^* \omega = \lambda \alpha$ for
\begin{displaymath}
\lambda \coloneqq \frac{\int_{S_n} f^*\omega^n}{\int _{S_n} f^* \omega^{n-1}\wedge \alpha} >0.
\end{displaymath}
As explained at the beginning of the proof, this implies that $f$ is indeed a totally geodesic immersion, as required.
\end{proof}

The ``double equidistribution computation'' employed above is reminiscent of an argument appearing in \cite{buiumpoonen1, mybuiumpoonen}, which was again motivated by the Zilber--Pink philosophy.

\begin{rmk}\label{rem:TGsubgroup}
Technically speaking, using \cite[Lem. 8.2]{BFMS2} one only obtains the fact that any totally geodesic divisor $Z$ is of the form $Z=\pi(\Gam_n\bs\Gam_n gL)$ for some $g\in\SU(1,n)$, where $L$ is an intermediate subgroup $W\subseteq L\subseteq N_G(W)$ and $gLg^{-1}\cap\Gam_n$ is a lattice in $gLg^{-1}$. Moreover, there are many instances where it is necessary to take $L$ strictly larger than $W$ in order for $gLg^{-1}\cap\Gam_n$ to be a lattice in $gLg^{-1}$. However, in the case where $\Gam_n$ is a lattice of simplest kind, one can always take $L=W$ which is what we are implicitly using in \Cref{lem:equidistribute}.
\end{rmk}

\section{The proof of \Cref{mainthmpre}}\label{sectionimmersion}

We now prove \Cref{mainthmpre} using some Hodge theoretic features behind the theory of Shimura varieties. This can serve as a warm up case for later arguments that use the full power of the theory of variations of Hodge structures, though it only requires the language used in previous sections, allowing us to defer the introduction of more technical definitions and results.

Suppose that $S_n=\Gamma _n \backslash \mathbb{B}^n$ is an arithmetic ball quotient and that
\begin{displaymath}
f: S_n \longrightarrow S_m \coloneqq \Gamma _m \backslash \mathbb{B}^m
\end{displaymath}
is an immersive holomorphic map to another arithmetic ball quotient. From now on, we always assume that $f(S_n)$ is \emph{monodromy generic}, i.e., is not contained in any strict complex totally geodesic subvariety $S^\prime$ of $S_m$. Let $\Gamma_f$ be the graph of $f$ in $S_n \times S_m$. We now prove the following equivalent formulation of \Cref{mainthmpre}.

\begin{theor}\label{thm11}
Let $\{Z_\ell\}$ (resp.\ $\{Y_\ell\}$) be an infinite sequence of distinct complex totally geodesic subvarieties of $S_n$ (resp.\ $S_m$). If $f(Z_{\ell})\subset Y_{\ell}$ for all $\ell$, then $f$ is a totally geodesic immersion.
\end{theor}

We first state the \emph{geometric Zilber--Pink conjecture}, which is a theorem in our special case. In fact in the more general setting of arbitrary Shimura varieties this is a theorem of Daw and Ren \cite{MR3867286} (see also the more general \Cref{geometricZP}, which will be used later). First, recall that an intersection of subvarieties is \emph{atypical} if its dimension is smaller than the expected dimension coming from a differential topology codimension count.

\begin{theor}[Special case of \cite{MR3867286} and \cite{2021arXiv210708838B}]\label{geomzp00}
Let $Y$ be a closed irreducible subvariety of $S_n\times S_m$. If $Y$ has a Zariski dense set of atypical intersections with totally geodesic subvarieties of positive dimension, then $Y$ is contained in a strict complex totally geodesic subvariety $S_0 \subset S_n\times S_m$.
\end{theor}

\begin{proof}[Proof of \Cref{thm11}]

Define $ U_\ell \coloneqq \Gamma_f \cap (Z_\ell \times Y_\ell)$. The key calculation in the proof is the following.

\begin{lem}
The intersection $U_{\ell}$ between $\Gamma_f$ and $Z_{\ell}\times Y_{\ell}$ is atypical for infinitely many $\ell$, that is,
\begin{equation}\label{eq4.2}
\codim_{S_n\times S_m}(U_\ell) < \codim_{S_n\times S_m}(\Gamma_f) + \codim_{S_n\times S_m}(Z_\ell \times Y_\ell).
\end{equation}
\end{lem}

\begin{proof}
We can pass to a subsequence and assume that the dimensions of $Z_{\ell}$ and $Y_{\ell}$ are independent of $\ell$. Set $d=\dim Z_{\ell}$ and $\dim Y_{\ell}= m-e$. We then have:
\begin{align*}
\codim_{S_n\times S_m}(U_\ell)&=n+m-d \\
\codim_{S_n\times S_m}(\Gamma_f) + \codim_{S_n\times S_m}(Z_\ell \times Y_\ell) &= m+(n+m-d-m+e) \\ &= n+m-d+e
\end{align*}
Since $e>0$, Equation \eqref{eq4.2} holds.
\end{proof}

Since we assumed that $f(S_n)$ is monodromy generic, we must prove that $f(S_n)=S_m$. Applying \Cref{geomzp00} to $\Gam_f \subset S_n\times S_m$, we conclude that $\Gamma_f$ is contained in a \textbf{strict} totally geodesic subvariety $S_0$ of $S_n\times S_m$. As $f(S_n)$ is monodromy generic, the second projection of $S_0$ on $S_m$ is surjective. Therefore $S_0$ is a strict totally geodesic subvariety of $S_n\times S_m$ dominating both factors. This is possible only when $S_n$ is commensurable with $S_m$ and $S_0$ is a modular correspondence. Since $\Gamma_f \subset S_0$ have the same dimension and are irreducible, it follows that $\Gamma_f=S_0$, i.e., $f$ is a totally geodesic embedding.
\end{proof}

Using a similar argument we can also prove the following, where a sequence of geodesics is \emph{generic} when they are not contained in a finite union of maximal totally geodesic subvarieties.

\begin{cor}\label{corcor}
Let $S_n\coloneqq\Gamma _n \backslash \mathbb{B}^n$ be a finite-volume ball quotient with an immersive holomorphic map
\begin{displaymath}
f: S_n \longrightarrow S_m \coloneqq \Gamma _m \backslash \mathbb{B}^m
\end{displaymath}
to another finite-volume ball quotient. Let $\{C_\ell\}$ be a generic sequence of distinct closed geodesics in $S_n$ for which there exists a strict complex totally geodesic subvariety $Y_\ell \subset S_m$ such that $f(C_\ell) \subset Y_\ell$ for all $\ell$. Then $f$ is a totally geodesic immersion.
\end{cor}

The proof of \Cref{corcor} uses the so called \emph{Baby Ax--Lindemann theorem} \cite[Prop. 2.6]{zbMATH06845358}, which we recall below in the special case we need.

\begin{prop}\label{babyALW}
Let $Z\subset \mathbb{B}^n$ be a real totally geodesic subvariety (of any real dimension). Let $\Gamma \subset \PU(1,n)$ be an arithmetic lattice, and $\pi : \mathbb{B}^n \to S_\Gamma=\Gamma \backslash \mathbb{B}^n $ be the associated locally Hermitian space. Then the Zariski closure of $ \pi (Z)$ is a complex totally geodesic subvariety of $S_\Gamma$.
\end{prop}
 
\begin{proof}[Proof of \Cref{corcor}]
Note that we may reduce to the case that $S_m$ is arithmetic. Indeed, if $S_m$ were nonarithmetic the fact that $\{C_\ell\}$ is generic implies that, after passing to maximal proper totally geodesic subvarieties, $f(S_n)$ is contained in a finite union of totally geodesic subvarieties. Since the map is a smooth immersion, the image is contained in a single such subvariety $S^\prime$. Replacing $S_m$ with $S^\prime$, we see that $S^\prime$ must have infinitely many distinct maximal totally geodesic subvarieties (see \cite{BFMS2} or \cite{bu}), and hence is arithmetic.

By assumption, $C_\ell\subset f^{-1}(Y_\ell)$. Since $f$ is in fact an algebraic morphism by Borel's extension theorem \cite{zbMATH03394558}, the right-hand side is Zariski closed and 
\begin{displaymath}
Z_\ell \coloneqq \overline{C_\ell} \subset f^{-1} (Y_\ell)\subsetneq S_n.
\end{displaymath}
Thanks to \Cref{babyALW}, $\overline{C_\ell}$ is a complex totally geodesic subvariety of $S_n$, which we denote by $Y_\ell$. Since the sequence $\{C_\ell\}$ is generic, we can apply \Cref{thm11} to conclude the proof.
\end{proof}

\section{Hodge theoretic preliminaries and a quick introduction to Zilber--Pink}\label{hodgeprel}

We begin this section by recalling the Andr\'{e}--Oort and Zilber--Pink conjectures and some formalism from Hodge theory. Loosely speaking, our results on (closed) geodesics are inspired by the analogy between closed geodesics and CM points\footnote{A point of a Shimura variety is called \emph{CM} if its associated Mumford--Tate group is a torus.} and the following.

\begin{theor}[Special case of Andr\'{e}--Oort in $S_1\times S_2$, {\cite{2021arXiv210908788P}}]\label{aoformaps}
A holomorphic map $f:S_1\to S_2$ between Shimura varieties sending a Zariski dense set of CM points of $S_1$ to CM points of $S_2$, is a morphism of Shimura varieties. In particular, it is totally geodesic.
\end{theor}

For this we refer the reader to \cite{MR3821177, 2021arXiv210708838B}.

\begin{rmk}
It may be interesting to compare this to what happens with abelian varieties. If $A,B$ are complex abelian varieties, any morphism $f: A \to B$ has totally geodesic image. This follows from a simple \emph{Rigidity lemma} as in \cite[Cor 1, page 43]{zbMATH03353452}. If $f$ sends a torsion point of $A$ to a torsion point of $B$, then the aforementioned corollary shows that $f$ is in fact a homomorphism. This latter statement is the analogue of \Cref{aoformaps} and shows that a version of Siu's immersion problem (\Cref{siuquestion}) holds for abelian varieties.
\end{rmk}

\subsection{Hodge theory and period domains}

We begin with some general notation and conventions. In what follows, an algebraic variety $S$ is a reduced scheme of finite type over the field of complex numbers, but may be reducible. If $S$ is an algebraic (resp.\ analytic) variety, by a subvariety $Y \subset S$ we always mean a \emph{closed} algebraic (resp.\ analytic) subvariety.

A $\Q$-Hodge structure of weight $n$ on a finite dimensional $\Q$-vector space $V$ is a decreasing filtration $F^\bullet$ on the complexification $V_\C$ such that
\[
V_\C= \bigoplus_{p\in \Z} F^{p} \oplus \overline{F^{n-p}}.
\]
The category of pure $\mathbb{Q}$-Hodge-structures is Tannakian and moreover semisimple when considering polarizable Hodge structures, as we frequently will. The Mumford--Tate group $\MT(V) \subset \GL(V)$ of a $\mathbb{Q}$-Hodge structure $V$ is the Tannakian group of the Tannakian subcategory $\langle V\rangle ^\otimes$ of $\mathbb{Q}$-Hodge structures generated by $V$. Equivalently, $\MT(V)$ is the smallest $\mathbb{Q}$-algebraic subgroup of $\GL(V)$ whose base-change to $\mathbb{R}$ contains the image of $h: \mathbb{S} \to \GL(V_{\mathbb{R}})$. It is also the stabilizer in $\GL(V)$ of the Hodge tensors for $V$. When $V$ is polarized, this is a reductive algebraic group.
 
In the two sections below we recall the setting and the main results from \cite{2021arXiv210708838B}. We refer also to the companion \cite{2023arXiv231211246B} (especially Section 2 in \emph{op. cit.}) for a quick introduction to this circle of ideas.

\subsection{Typical and atypical intersections}

To understand a VHS $\VV \to S$, we consider the associated holomorphic period map 
\begin{equation} \label{period0}
 \Phi: S^{\an} \longrightarrow \Gamma \backslash D,
\end{equation}
which completely describes $\VV$. We let $(\G, D)$ denote the generic Hodge datum of $\VV$ and $\Gamma \backslash D$ the associated Hodge variety. The Mumford--Tate domain $D$ decomposes as a product $D_1 \times \cdots \times D_k$ according to the decomposition of the adjoint group $\G^\ad$ into a product $\G_1 \times \cdots \times \G_k$ of simple factors, where some $\G_i$ may be $\RR$-anisotropic. Replacing $S$ by a finite \'etale cover and reordering the factors if necessary, the lattice $\Gamma \subset\G^\ad(\R)^+$ decomposes as a direct product $\Gamma \cong \Gamma_1 \times \cdots \times \Gamma_r$ with $ r \leq k$, where $\Gamma_i\subset \G_i(\R)^+$ is an arithmetic lattice for each $i$. Writing $D_{\mathrm{triv}} = D_{r+1} \times \cdots \times D_k$ for the product of factors where the monodromy is trivial, which includes all of the factors $D_i$ for which $\G_i$ is $\RR$-anisotropic, the period map can be written as
\begin{equation} \label{period}
\Phi: S^{\an} \longrightarrow \Gamma \backslash D \cong \Gamma_1 \backslash D_1 \times \cdots \times \Gamma_r \backslash D_r \times D_{\mathrm{triv}},
\end{equation}
where the projection of $\Phi(S^\an)$ to $D_{\mathrm{triv}}$ is a point.

We distinguish between special subvarieties $Z$ of \emph{zero period dimension}, which are geometrically elusive, and those of \emph{positive period dimension}.

\begin{defi} \label{positive period dimension} \label{fpositive} \hfill
 \begin{enumerate}
  \item
 A subvariety $Z$ of $S$ is said to be of {\em positive period dimension for $\VV$} if $\Phi(Z^{\an})$ has positive dimension, that is, if $\dim_\C\Phi(Z^\an)>0$.
 \item The {\em Hodge locus of positive period dimension}, $\HL(S, \VV^\otimes)_{\pos}$, is the union of the special subvarieties of $S$ for which $\VV$ has positive period dimension.
\end{enumerate}
\end{defi}

Using period maps, special subvarieties can also be defined as {\em intersection loci}. Indeed, a closed irreducible subvariety $Z \subset S$ is special for $\VV$ precisely when $Z^{\an}$ coincides with an analytic irreducible component $\Phi^{-1}(\Gamma^\prime \backslash D^\prime)^0$ of $\Phi^{-1}(\Gamma^\prime \backslash D^\prime)$, for $(\G^\prime, D^\prime) \subset(\G, D)$ the generic Hodge subdatum of $Z$ and $\Gamma^\prime \backslash D^\prime\subset \Gamma \backslash D$ the associated Hodge subvariety. We will equivalently say that $Z$ is special for $\Phi$.

\begin{defi}\label{atypical}
 Let $Z = \Phi^{-1}(\Gamma^\prime \backslash D^\prime)^0\subset S$ be a special subvariety for $\VV$ with generic Hodge datum $(\G^\prime, D^\prime)$. Then $Z$ is said to be \emph{atypical} if $\Phi(S^{\an})$ and $\Gamma^\prime \backslash D^\prime$ do not intersect generically along $\Phi(Z^\an)$. That is, $Z$ is atypical when
 \begin{equation} \label{equation atypical}
  \codim_{\Gamma\backslash D} \Phi(Z^{\an}) < \codim_{\Gamma\backslash D} \Phi(S^{\an}) + \codim_{\Gamma\backslash D} \Gamma^\prime\backslash D^\prime.
 \end{equation}
Otherwise $Z$ is said to be \emph{typical}. The \emph{atypical Hodge locus} $\HL(S,\VV^\otimes)_{\atyp} \subset \HL(S, \VV^\otimes)$ (resp.\ the \emph{typical Hodge locus} $\HL(S,\VV^\otimes)_{\typ} \subset \HL(S, \VV^\otimes)$) is the union of the atypical (resp.\ strict typical) special subvarieties of $S$ for $\VV$.
\end{defi}

\begin{rmk}\label{rmkonnormalizations}
Since it will be important in the sequel, we remark here that the notion of typicality and atypicality is always computed with respect to the smallest special subvariety of $\Gamma \backslash D$, usually called the special closure. It can, and does, happen that 
\begin{displaymath}
Z=\Phi^{-1}(\Gamma^\prime \backslash D^\prime)^0=\Phi^{-1}(\Gamma'' \backslash D'')^0
\end{displaymath}
for some bigger $\Gamma'' \backslash D''$, but computing codimensions with the latter would give the wrong conclusions.
\end{rmk}

Let $\VV$ be a polarizable $\ZZ$VHS on an irreducible smooth quasi-projective variety $S$, from \cite{2021arXiv210708838B} we expect:

\begin{conj}[Zilber--Pink conjecture for the atypical Hodge locus, strong version] \label{main conj}
The atypical Hodge locus $\HL(S,\VV^\otimes)_{\atyp}$ is a finite union of atypical special subvarieties of $S$ for $\VV$.
\end{conj}

\begin{conj}[Density of the typical Hodge locus] \label{conj-typical}
 If $\HL(S, \VV^\otimes)_\typ$ is not empty then it is analytically dense in $S$.
\end{conj}

\subsection{Some results on the distribution of the Hodge locus}

We now recall some results from \cite{2021arXiv210708838B}, namely Theorem 6.1, Theorem 10.1, and Remark 10.2 therein. To simplify the statements the reader can first assume that the generic Mumford--Tate group is simple, however later it will be used in a product situation.

\begin{theor}[Geometric Zilber--Pink]\label{geometricZP}
Let $\VV$ be a polarizable $\ZZ$VHS on a smooth connected complex quasi-projective variety $S$, with generic Hodge datum $(\G, D)$ and $Z$ be an irreducible component of the Zariski closure of the union of the atypical special subvarieties of positive period dimension in $S$. Then either
\begin{itemize}
 \item[(a)] $Z$ is a maximal atypical special subvariety, or
  \item[(b)] the adjoint Mumford--Tate group $\G_Z^\ad$ decomposes as a nontrivial product $\HH^\ad_Z \times \LL_Z$, $Z$ contains a Zariski-dense set of fibers of $\Phi_{\LL_{Z}}$ which are atypical weakly special subvarieties of $S$ for $\Phi$, where (possibly up to an \'{e}tale covering)
\[
\Phi_{|Z^\an}= (\Phi_{\HH_{Z}}, \Phi_{\LL_{Z}}): Z^\an \longrightarrow \Gamma_{\G_{Z}}\backslash D_{\G_{Z}}= \Gamma_{\HH_{Z}}\backslash D_{\HH_{Z}} \times \Gamma_{\LL_{Z}}\backslash D_{\LL_{Z}} \subset \Gamma \backslash
D,
\]
and $Z$ is Hodge generic in a special subvariety $\Phi^{-1}(\Gamma_{\G_{Z}}\backslash D_{\G_{Z}})^0$ of $S$ for $\Phi$ which is monodromically typical and therefore typical.
\end{itemize}
\end{theor}
We refer also to \cite[Thm. 7.1]{2024arXiv240616628B} for an effective proof of a more general statement (proved using differential geometry rather than $o$-minimality).
\begin{theor}\label{typicallocus}
If the typical Hodge locus $\HL(S,\VV^\otimes)_{\pos,\typ} $ is nonempty then $\HL(S,\VV^\otimes)_{\pos,\typ}$ is analytically dense in $S$.
\end{theor}

The above statement already hides an application of \Cref{geometricZP}, as explained in \cite[Rmk.\ 10.2]{2021arXiv210708838B}. Indeed it is proved by first showing that $\HL(S,\VV^\otimes)_{\pos}$ in dense then invoking \Cref{geometricZP} to say that $\HL(S,\VV^\otimes)_{\pos,\atyp}$ is algebraic, and therefore is exactly $\HL(S,\VV^\otimes)_{\pos,\typ}$, which is dense. In future applications we will actually use this finer version. See also \cite[Thm.\ 1.6 and Rmk.\ 1.7]{2023arXiv230316179K} for a related discussion.

\begin{rmk}
Hodge theory actually gives a simple combinatorial criterion to decide whether $\HL(S,\VV^\otimes)_{\typ}$ is empty or not. Indeed see the recent work \cite{2022arXiv221110592E, 2023arXiv230316179K}.
\end{rmk}

\section{General Hodge theoretic statement: from typical to atypical intersections}\label{sectionhodge}

\subsection{Isogenies between VHSs}

For an introduction to Tannakian categories, we refer for example to the article of Deligne and Milne \cite{zbMATH03728195}. The category of $\Q$VHS on a smooth quasi-projective base $S$, which we denote by $\Q VHS/S$, is Tannakian (we can fix a fiber functor $\omega_s$ corresponding to some base point $s\in S$). Given a $\Z VHS$ $\VV$, we denote the associated $\Q VHS$ by $\VV_\Q$.

\begin{defi}\label{def:isog}
Let $\mathbb{V}_1, \VV_2\in \Z VHS/S$. We say that $\VV_1$ and $\VV_2$ are \emph{isogenous} if there is an equivalence of tensor categories $\langle \VV_{1,\Q} \rangle^\otimes \cong \langle \VV_{2,\Q} \rangle^\otimes$, where $ \langle \VV_{i,\Q} \rangle^\otimes$ denotes the smallest Tannakian subcategory of $\Q VHS$'s containing $\VV_{i,\Q}$.
\end{defi}
Of course the Tannakian categories $\langle \VV_{i,\Q} \rangle^\otimes$ appearing above are equivalent (as tensor categories) to the category of finite dimensional representations of their generic Mumford-Tate group. (The equivalence is realized by the functor of $\otimes$-automorphisms of the fiber functor). It can happen that two VHS have isomorphic Mumford-Tate groups, but the isomorphism does not induce an equivalence of tenor categories. C.f. Def. 1.10 and Thm. 2.11 in the article of Deligne and Milne\cite{zbMATH03728195}.
\begin{rmk}\label{rmkonabelian}
If two complex principally polarized abelian varieties $A,B$ that are isogenous in the usual sense (either via a polarized or an unpolarized isogeny), then the Hodge structures $H^1(A,\Z), H^1(B,\Z)$ also isogenous in the sense of \Cref{def:isog} (here we take as base $S$ the spectrum of $\C$). However the converse is in general not true, for example $H^1(A,\Z)$ is isogenous to $H^1(A\times A, \Z)$. However, for two principally polarized $g$-dimensional abelian varieties with Mumford--Tate group $\mathbf{GSp}_{2g}$ the two notions agree, since this is the case when the Mumford--Tate group is as big as possible.
\end{rmk}

Note that isogeny is an equivalence relation which we denote by $\VV_1\simeq \VV_2$. Let $\Gamma_1\backslash D_1$ and $\Gamma_2\backslash D_2$ be two Mumford--Tate Domains. A \emph{modular correspondence} between $\Gamma_1\backslash D_1$ and $\Gamma_2\backslash D_2$ is a subvariety 
\[
V\subset \Gamma_1\backslash D_1\times \Gamma_2\backslash D_2
\]
such that the two projections $p_1: V\rightarrow \Gamma_1\backslash D_1$ and $p_2: V\rightarrow \Gamma_2\backslash D_2$ are finite \'etale surjective maps (see also \cite[Def. 3.23]{2021arXiv210708838B} for more details). In this situation $V$ is a Mumford--Tate domain, that is $V=\Gamma\backslash D$, we have that $D\simeq D_1\simeq D_2$ and that $\Gamma_1$ is commensurable with $\Gamma_2$. The following proposition translates the isogeny relation from a Tannakian statement to a more useful one on period maps.

\begin{prop}\label{charisog}
Let $\VV_1, \VV_2$ be $\Z VHS$ on $S$ with associated Hodge data $(\G_i,D_i)$ and period map $\psi _i$.
The following are equivalent:
\begin{enumerate}
\item There is an isogeny between $\VV_1$ and $\VV_2$;
\item The Hodge data $(\G_1,D_1)$, $(\G_2,D_2)$ are isomorphic. Using $(\G,D)$ for the corresponding Hodge datum and $\psi_1: S\to \Gamma_1 \backslash D$, $\psi_2: S\to \Gamma_2 \backslash D$ for the period maps associated with $\VV_1$, $\VV_2$ (respectively), then there exists a modular correspondence 
\[
\Gamma\backslash D\subset \Gamma_1\backslash D \times \Gamma_2\backslash D
\]
and a period map $\psi_0: S\to \Gamma\backslash D$ such that if $p_i: \Gamma\backslash D \to \Gamma_i\backslash D $ are the natural projections, then $\psi_i= p_i \circ \psi_0$ for $i\in\{1,2\}$.
\end{enumerate}
\end{prop}

\begin{proof}
Let $\VV$ be a $\ZZ$VHS on $S$ with generic Hodge datum $(\G,D)$. Then the Tannakian category generated by $\VV_{\QQ}$ is determined by $(\G,D)$, as any $\WW$ in $ \langle \VV_{\Q} \rangle^\otimes$ is, by definition of a Tannakian category, given by a representation of $\G$ on the fibre $W_{\QQ}$ of $\WW_{\QQ}$ at a point $s\in S$ and the Hodge structure on $W_{\QQ}$ is determined by some Hodge morphism $\alpha: \bS\rightarrow \G$ belonging to $D$. Therefore if $(\G_1,D_1)\cong (\G_2,D_2)$, then $\VV_1, \VV_2$ are isogenous.

If $\VV_1, \VV_2$ are isogenous, they have isomorphic generic Mumford--Tate groups, denoted by $\G$, under the interpretation of the Mumford--Tate group in terms of Tannakian categories recalled above. Thus
\begin{displaymath}
 \langle \VV_{1,\Q} \rangle^\otimes\cong \langle \VV_{2,\Q} \rangle^\otimes \cong \langle \VV_{1,\Q} \oplus \VV_{2,\Q} \rangle^\otimes
\end{displaymath}
and therefore the generic Mumford--Tate group of $\VV_1\oplus \VV_2$ is also $\G$.

The $\ZZ$VHS $\VV_1\oplus\VV_2$ corresponds to a period map
\[
\psi_0: S\longrightarrow \Gamma\backslash D\subset \Gamma_1\backslash D_1 \times \Gamma_2\backslash D_2,
\]
where $(\G,D)$ is the generic Hodge subdatum associated with $\VV_1\oplus\VV_2$. Moreover, if the natural projections are $p_i: \Gamma\backslash D \to \Gamma_i\backslash D_i$, we have $\psi_i= p_i \circ \psi_0$. Let $\alpha_i:\bS\to \G$ be Hodge morphisms such that $D_i=\G(\RR)^+\alpha_i$. Then there exists an embedding of $\G$ in $\G \times \G$ such that $D$ is the $\G(\RR)^+$-orbit of $\alpha=(\alpha_1,\alpha_2)$. Therefore $D\simeq D_1\simeq D_2$.
\end{proof}

\subsection{The main statement}

\begin{theor}\label{hodgethm}
Let $S$ be a smooth quasi-projective variety and $\mathbb{V}_1, \mathbb{V}_2$ two pure polarized $\Z$VHSs on $S$. Assume that the generic Mumford--Tate groups of $\VV_1$ and $\VV_2$ are $\QQ$-simple.
If 
\[
\HL(S,\mathbb{V}_1^\otimes)_{\pos, \typ}=\HL(S,\mathbb{V}_2^\otimes)_{\pos, \typ}\neq \emptyset,
\]
 then $\mathbb{V}_1$ is isogenous to $\mathbb{V}_2$. As a consequence, $\HL(S,\mathbb{V}_1^\otimes)=\HL(S,\mathbb{V}_2^\otimes)$.
\end{theor}

Some stronger variants with the \emph{weakly special Hodge locus} replacing the Hodge locus can be given. Also, the same proof works if there exits a Zariski dense subset of components of $\HL(S,\mathbb{V}_1^\otimes)_{\pos, \typ}$ that is contained in $\HL(S,\mathbb{V}_2^\otimes)_{\pos, \typ}$.

\begin{rmk}
By \Cref{typicallocus}, the condition $ \HL(S,\mathbb{V}_i^\otimes)_{\pos, \typ}\neq \emptyset$ implies that $\HL(S,\mathbb{V}_i^\otimes)_{\pos, \typ}$ is analytically dense in $S$. \Cref{conj-typical} predicts that $ \HL(S,\mathbb{V}_i^\otimes)_{\typ}$ should be analytically dense whenever it is nonempty. The condition $\HL(S,\mathbb{V}_1^\otimes)_{\pos, \typ}=\HL(S,\mathbb{V}_2^\otimes)_{\pos, \typ}\neq \emptyset$ is therefore quite restrictive. Moreover by \cite[Thm. 1.5]{2021arXiv210708838B}, the level of the $\ZZ$VHS $\VV_1$ and $\VV_2$ must be $1$ (the Shimura case) or $2$.
\end{rmk}

To better explain \Cref{hodgethm}, we give a simple and more concrete corollary about two principally polarized families of abelian varieties $h_i:\mathcal{A}_i \to S$.

\begin{cor}
Fix $g \geq 2$. Let $S$ be a smooth quasi projective variety and $h_i:\mathcal{A}_i \to S$ be two principally polarized $g$-dimensional families of abelian varieties whose monodromy group is $\mathbf{Sp}_{2g}$. Set $\VV_i\coloneqq{R^1h_{i ,*}\Z}$ for the associated VHS. If
\begin{equation}
\HL(S,\VV_1^\otimes)_{\pos, \typ}=\HL(S,\VV_2^\otimes)_{\pos, \typ}\neq \emptyset,
\end{equation}
then $\mathcal{A}_1$ is isogenous to $\mathcal{A}_2$ as principally polarized $S$-abelian schemes.
\end{cor}

\begin{proof}
We will use \Cref{rmkonabelian} and the following fact: Given two principally polarized families of abelian varieties $h_i:\mathcal{A}_i \to S$, one has $\Hom_S (\mathcal{A}_1, \mathcal{A}_2)= \Hom (R_1{h_1,_*}\Z,R_1{h_2,_*}\Z)$, where the second $\Hom$ is in the category of $\Z$VHS; see \cite[Cor.\ 4.4.15]{MR0498551}. Indeed \Cref{hodgethm} implies that the two $\Z$VHSs $\VV_i$ for $i\in\{1,2\}$
are isogenous in the sense of \Cref{def:isog}, and the above facts imply that this notion of isogeny translates precisely to the traditional one.
\end{proof}

\begin{rmk}The above statement bears some similarities with work of Baker and DeMarco \cite{zbMATH05941080} in complex dynamics, which shows that for any fixed $a,b \in \C$ and any integer $d \geq 0$, the set of $c\in \C$ for which both $a$ and $b$ are preperiodic for $z^d +c$ is infinite if and only if $a^d=b^d$. Also see generalizations like in the work of Yuan and Zhang \cite{yz}. Both proofs mentioned above use equidistribution but here we work with positive dimensional subvarieties. The statement without \emph{pos} is deeper and would follow from \Cref{main conj}.
\end{rmk}

\begin{proof}[Proof of \Cref{hodgethm}]
The main step in the proof is the following lemma, where we do not need to assume that the generic Mumford--Tate group of $\VV_i$ is $\QQ$-simple.

\begin{lem}\label{lalala}
Assuming the hypothesis of \Cref{hodgethm}, let $f_i: S\to \Gamma_i \backslash D_i$ be the period map associated with $\mathbb{V}_i$, $i=1,2$, and let
\begin{displaymath}
f_1\times f_2: S \longrightarrow \Gamma_1 \backslash D_1 \times \Gamma_2 \backslash D_2
\end{displaymath}
be the period map associated with $\VV_1\oplus \VV_2$. Then $\HL(S,(\VV_1\oplus\VV_2)^{\otimes})_{\pos, \atyp}$ is analytically dense in $S$.
\end{lem}

\begin{proof}
As $\HL(S,\VV_i^{\otimes})_{\pos,\typ}\neq \emptyset$, by \Cref{typicallocus} there exists sequences of special subvarieties $\{Z_{i,\ell}\}_{\ell\in \NN}$
 in $\Gamma_i\backslash D_i$ and a sequence $\{W_{\ell}\}_{\ell \in \NN}$ of components of both $\HL(S,\VV_1^{\otimes})_{\pos,\typ}$ and $\HL(S,\VV_2^{\otimes})_{\pos,\typ}$, such that
\begin{displaymath}
W_\ell= f_1^{-1}(Z_{1,\ell})^0=f_2^{-1}(Z_{2,\ell})^0
\end{displaymath}
for some components $f_i^{-1}(Z_{i,\ell})^0$ of $f_i^{-1}(Z_{i,\ell})$ and for which $\cup_\ell W_\ell$ is Zariski dense in $S$. The fact that the $W_\ell$ are again typical intersections, guaranteed by \Cref{typicallocus}, implies that $W_\ell$ is Hodge generic for $\VV_i$ in $Z_{i,\ell}$ (see also the discussion in \Cref{rmkonnormalizations}). By passing to a subsequence, we assume that $Z_{i,\ell}$ has fixed dimension $z_i$.

Since $W_{\ell}$ is realized in two ways as a typical intersection, we have
\begin{displaymath}
\codim_{\Gamma_i \backslash D_i} f_i(W_\ell) = \codim_{\Gamma_i \backslash D_i} Z_{i,\ell} + \codim_{\Gamma_i \backslash D_i}(f_i(S) ).
\end{displaymath}
for each $i\in\{1,2\}$.
That is, if $w_i\coloneqq \dim f_i(W)$, $d_i \coloneqq \dim \Gamma_i \backslash D_i$, and $s_i \coloneqq \dim f_i(S) $, then
\begin{equation}\label{eqtyp}
d_i - w_i=d_i -z_i+ d_i -s_i.
\end{equation}
Note that we do not assume that our period maps are immersive. Notice that 
\begin{displaymath}
W_\ell \subset (f_1\times f_2) ^{-1} (Z_{1,\ell} \times Z_{2,\ell}) \varsubsetneq S.
\end{displaymath}
Let $W_\ell^\prime$ be an irreducible component of $(f_1\times f_2)^{-1} (Z_{1,\ell} \times Z_{2,\ell})$ containing $W_\ell$. We claim that $W_\ell^\prime$ is an atypical intersection of positive period dimension in the Hodge locus $\HL(S,(\VV_1\oplus\VV_2)^{\otimes})$ for all $\ell$. For this, we need to show that
\begin{displaymath}
\codim_{ \Gamma_1 \backslash D_1 \times \Gamma_2 \backslash D_2}((f_1 \times f_2)(W_\ell^\prime)) < \codim_{ \Gamma_1 \backslash D_1 \times \Gamma_2 \backslash D_2}(Z_{1,\ell}\times Z_{2,\ell}) + \codim_{ \Gamma_1 \backslash D_1 \times \Gamma_2 \backslash D_2}((f_1\times f_2)(S)).
\end{displaymath}
Summing Equation \eqref{eqtyp} for $i=1,2$ we have
\begin{equation}
d_1+d_2 - w_1-w_2=d_1+d_2 -z_1-z_2+ d_1+d_2 -s_1 -s_2.
\end{equation}
This implies the requisite equation by noticing that
\begin{displaymath}
\max \{\dim (f_1 (Y)),\dim (f_2 (Y)) \} \leq \dim ((f_1 \times f_2) (Y)) \leq \dim (f_1 (Y))+\dim (f_2 (Y))
\end{displaymath}
for any $Y$ in $S$.
\end{proof}

\begin{lem}
Each of the following hold:
\begin{enumerate}
\item If the monodromy of $\VV_i$ is equal to the derived subgroup of its Mumford--Tate group $\G_i$, then
$(f_1\times f_2)(S) $ is not Hodge generic in $\Gamma_1 \backslash D_1 \times \Gamma_2 \backslash D_2$.

\item If the Mumford--Tate group $\G_i$ of $\VV_i$ is $\QQ$-simple, then $\G_1^{\ad}\cong\G_2^{\ad}$, $D_1\simeq D_2$. Denoting the corresponding Hodge datum by $(\G,D)$, we have that $(f_1\times f_2)(S) $ is Hodge generic in a modular correspondence
\[
\Gamma\backslash D\subset \Gamma_1 \backslash D \times \Gamma_2 \backslash D.
\]
\end{enumerate}
\end{lem}

\begin{proof}
By \Cref{lalala}, $S$ has a Zariski dense set of positive dimensional atypical components of the Hodge locus for the period map $f_1\times f_2$. Therefore $S$ satisfies conclusion (a) or (b) of \Cref{geometricZP}. If conclusion (a) holds, then $(f_1\times f_2)(S)$ is not Hodge generic in $\Gamma_1 \backslash D_1 \times \Gamma_2 \backslash D_2$, as desired. If the monodromy of $\VV_i$ is equal to the derived subgroup of its Mumford--Tate group $\G_i$, then the projection of $f_i(S)$ on every factor of $\Gamma_i\backslash D_i$ is positive dimensional and Hodge generic. The same is true for the period map $f_1\times f_2$. Therefore conclusion (b) is not satisfied. This finishes the proof of the first part of the lemma.

If the Mumford--Tate group $\G_i$ of $\VV_i$ is $\QQ$-simple, then the monodromy of $\VV_i$ is equal to the derived subgroup of its Mumford--Tate group $\G_i$ and by the first part, $(f_1\times f_2)(S)$ is not Hodge generic in $\Gamma_1 \backslash D_1 \times \Gamma_2 \backslash D_2$. Moreover in this situation, the only strict special subvarieties of $\Gamma_1 \backslash D_1 \times \Gamma_2 \backslash D_2$ whose projections on both factors are surjective are the modular correspondences. This forces $D_1\simeq D_2$ and $\G_1^{\ad}\cong\G_2^{\ad}$. This finishes the proof of the second part of the lemma.
\end{proof}

\Cref{hodgethm} is now a consequence of the characterization of isogenous $\ZZ$VHSs given in \Cref{charisog}.
\end{proof}

\begin{rmk}\label{rmkimplication} \Cref{hodgethm} (or, more precisely, a slight generalization of it) implies \Cref{thm11} for ball quotients, which we now explain. On the smaller ball quotient $S_n$, which we assume to be arithmetic for simplicity,
 we have a $\Z$VHS $\mathbb{V}_1$ induced by any faithful linear representation of $\operatorname{GU}(1,n)$. The Hodge locus for such a VHS does not depend on the linear representation by a theorem of Deligne and it is equal to the union of all sub-Shimura varieties. With the terminology introduced in \Cref{atypical}, a sub-Shimura variety is always in the typical Hodge locus
\begin{displaymath}
\HL(S_n,\mathbb{V}_1^\otimes)_{\pos}=\HL(S_n,\mathbb{V}_1^\otimes)_{\pos, \typ}
\end{displaymath}
and this locus is equal to the union of all totally geodesic subvarieties. Now a map $f: S_n \to S_m$ gives another VHS $\mathbb{V}_2$ by pulling back the natural one on $S_m$. Once this translation from immersions to Hodge theory is made, the proof of \Cref{thm11} is just a special case of the one given in \Cref{hodgethm}.
\end{rmk}

\section{Proofs of \Cref{mainthm12} and \Cref{richvhs0}}\label{mainthmonvhs}

In this section we prove the second part of \Cref{mainthm12} and then \Cref{richvhs0} following a similar argument. Let $\Gamma \subset \operatorname{PU}(1,n)$ for $n>1$ be a lattice of the simplest kind, which we can assume to be neat (i.e., no element of $\Gam$ has a nontrivial root of unity as an eigenvalue under the adjoint representation). Consider the smooth quasi-projective variety $S_\Gamma$ and let 
\begin{displaymath}
\rho: \Gamma \longrightarrow \HH(k)
\end{displaymath}
be a cohomologically rigid representation, i.e., $H^1(\Gamma , \operatorname{Ad}_\HH \rho)=0$, where
\begin{displaymath}
\Ad_\HH: \HH(k) \longrightarrow \Aut(\mathfrak{h})
\end{displaymath}
is the adjoint representation over some local field $k$. Assuming that $\rho$ is geodesically rich in the sense of \Cref{richdef}, we would like to show that $\rho$ extends. The proof is a generalization of the arguments from \cite[Sec.\ 4 and 5]{bu}, improved by combining them with \Cref{sectionhodge}.

\begin{proof}[Proof of \Cref{mainthm12} (2)]
The proof proceeds in two steps. The first is to build a rich $\Z$VHS, using the rigidity assumption, and the second is to use richness to finish the proof.

From the vanishing $H^1(\Gamma, \operatorname{Ad}_ \HH\rho)=0$ we have that $\rho$ underlies a $\C$VHS. This is essentially due to Simpson and is explained in detail in \cite[Sec.\ 4]{bu}. Moreover,
\begin{displaymath}
\Z (\rho)\coloneqq\Z [\operatorname{tr} \operatorname{Ad}_ \HH (\rho (\gamma)) : \gamma \in \Gamma]
\end{displaymath}
has a natural embedding in $\mathcal{O}_K$, the ring of integers of a totally real number field $K$. There are two ways to show that $\Z (\rho) $ lies in $\Oo_K$, as opposed to the ring of $S$-integers for some finite set of finite places $S$ of $K$. In \cite[Sec. 4]{bu} it is done by invoking the recent work of Esnault and Groechenig \cite{MR3874695}. However, here we can give a self-contained proof by applying \Cref{cor:Integral}.

From here forward we assume that $\rho (\Gamma)$ is Zariski dense in the $k$-group $\HH$ and will then prove that $\bfH(K)\cong G$. In particular, from the above we obtain a natural $K$-form of $\HH$ (that we denote by the same symbol $\HH$). Let $\sigma_i : K \to \R$ be the places of $K$, ordered in such a way that $\sigma_1$ is the identity for the given lattice embedding. Let $\widehat{\HH}\subseteq\operatorname{Res}_{K/\Q} \mathbf{H}$ and consider the $\Z$-local system
\begin{align*}
\hat{\rho}: \Gamma &\longrightarrow \widehat\HH (\R) \\
 \gamma &\longmapsto \hat{\rho} (\gamma)=(\sigma_i(\rho(\gamma))).
\end{align*}
Again by cohomological rigidity of $\rho$, we have that $\hat{\rho}$ naturally underlies a $\Z$VHS $\widehat{\mathbb{V}}$ (compare with \cite[Thm.\ 4.2.4]{bu}). As recalled in \Cref{hodgeprel}, $\widehat{\mathbb{V}}$ corresponds to a period map
\begin{displaymath}
\psi : S_{\Gamma}^{\an}\longrightarrow \widehat\HH (\Z) \backslash D.
\end{displaymath}
Now that we have a $\Z$VHS, as in \Cref{rmkimplication} we use the results of \Cref{sectionhodge} to prove the theorem. Let $S_\Gamma= \Gamma \backslash \mathbb{B}^n$ be the associated ball quotient, and $\mathbb{V}_1$ the $\Z$VHS on $S_\Gamma$ associated with some faithful linear rational representation of $\mathbf{G}$. By construction, $\HL(S_\Gamma, \mathbb{V}_1^{\otimes})_{\pos}$ is the union of the complex totally geodesic subvarieties of $S_{\Gamma}$, and it is equal to $\HL(S_\Gamma, \mathbb{V}_1^{\otimes})_{\pos,\typ}$; in particular it is nonempty.

It is straightforward to see that the fact that $\rho$ is complex rich implies that $\hat{\rho}$ is complex rich as well. Indeed, $S_{\Gamma_i}\subset S_\Gamma$ be a totally geodesic subvariety such that $\mathbf{H}_i=\overline{\rho(\Gamma_i)}$ has dimension smaller than $\dim \mathbf{H}$. Then $\dim \widehat{\mathbf{H}}_i:=\overline{\hat{\rho}(\Gamma_i)} < \dim \widehat{\mathbf{H}}$. Moreover $S_{\Gamma_i}$ lies in the $ \HL(S_\Gamma, \mathbb{V}_1^{\otimes})_{\pos,\typ}$, as well as $\HL(S_\Gamma, \widehat{\mathbb{V}}^{\otimes})_{\pos,\typ}$. Since this is true for a sequence of $i$ (the proof of) \Cref{hodgethm} implies that $\mathbb{V}_1,\widehat{\mathbb{V}} $ are isogenous which means that the period domain $\widehat\HH (\Z) \backslash D$ is isomorphic to $S_\Gamma$ and that $\psi$ is isogenous to the identity (as in \Cref{charisog}). Therefore $\hat\rho$ extends, which is only possible if the initial representation $\rho$ extends. That is, $\psi$ is an isomorphism. This concludes the proof of the desired statement.
\end{proof}

\begin{proof}[Proof of \Cref{richvhs0}]
The proof follows the same argument given above. The only difference is we are given a $\Q$VHS $\VV$ on the algebraic variety $S_\Gamma$. By applying \Cref{cor:Integral}, $\VV$ must be a $\Z$VHS, and the richness assumption implies that there is a generic sequence of subvarieties $S_{\Gamma_i}\in \HL(S_\Gamma, \VV^{\otimes})_{\pos,\typ}\cap \HL(S_\Gamma, \VV_1^\otimes)_{\pos,\typ}$. Therefore the results follows once again from \Cref{hodgethm}.
\end{proof}

\bibliographystyle{abbrv}
\bibliography{biblio.bib}

\Addresses

\end{document}